\documentclass[11pt]{amsart}
\usepackage{verbatim, graphicx, rotating, rotfloat, lscape}
\usepackage{morefloats}

\usepackage{url}
\usepackage{hyperref}
\usepackage{amssymb}
\usepackage{multirow}

\input xy
\xyoption{all}
\input epsf
\newlength{\tabwidth}
\newlength{\tabheight}
\newlength{\tabrule}
\newlength{\tabwidthx}
\newlength{\tabheightx}

\def\gentabbox#1#2#3#4{\vbox to \tabheight{\setlength{\tabrule}{#3}%
  \setlength{\tabwidthx}{#1\tabwidth}\addtolength{\tabwidthx}{\tabrule}%

\setlength{\tabheightx}{#2\tabheight}\addtolength{\tabheightx}{-\tabheight}%
  \hbox to #1\tabwidth{%
    \hspace{-0.5\tabrule}\rule{\tabrule}{#2\tabheight}\hspace{-\tabrule}%
    \vbox to #2\tabheight{\hsize=\tabwidthx%
      \vspace{-0.5\tabrule}\hrule width\tabwidthx height\tabrule%
      \vspace{-0.5\tabrule}\vfil%
      \hbox to \tabwidthx{\hss#4\hss}%
        \vfil\vspace{-0.5\tabrule}%
      \hrule width\tabwidthx height\tabrule\vspace{-0.5\tabrule}}%
    \hspace{-\tabrule}\rule{\tabrule}{#2\tabheight}\hspace{-0.5\tabrule}}%
  \vspace{-\tabheightx}}}
\def\genblankbox#1#2{\vbox to \tabheight{\vfil\hbox to
#1\tabwidth{\hfil}}}

\newcommand{\excise}[1]{}%{$\star$\textsc{#1}$\star$}

\newcommand{\field}{\mathbb}
\newcommand{\liealgebra}{\mathfrak}
\newcommand{\la}{\liealgebra}

\newcommand{\C}{{\field C}}

\newcommand{\N}{{\mathbb N}}
\renewcommand{\b}{\liealgebra b}

\newcommand{\n}{{\la n}}

\newtheorem{prop}{Proposition}[section]

\newtheorem{lemma}[prop]{Lemma}
\newtheorem{theorem}[prop]{Theorem}

\newtheorem{corollary}[prop]{Corollary}

\theoremstyle{definition}

\newtheorem{remark}[prop]{Remark}
\newtheorem{example}[prop]{Example}

\newtheorem{case}{Case}
\newtheorem{subcase}{Case}
\newtheorem{subsubcase}{Case}
\newtheorem{subsubsubcase}{Case}
\newtheorem{question}[prop]{Question}
\numberwithin{subcase}{case}
\numberwithin{subsubcase}{subcase}
\numberwithin{subsubsubcase}{subsubcase}

\newtheorem{definition}[prop]{Definition}

\newcommand{\caI}{\mathcal{I}}

\newcommand{\caO}{\mathcal{O}}

\begin{document}
\title[Bruhat order on clans]{The Bruhat order on clans}

\author{Benjamin J. Wyser}
\date{\today}

%\thanks{}
%\address{}
%\email{}

\begin{abstract}
We give an explicit description of the closure containment order (or ``Bruhat order'') on the set of
orbits of $GL_p \times GL_q$ on the flag variety $GL_{p+q}/B$, relative to the parametrization of the orbits by
combinatorial objects called ``clans''.  This leads to a corresponding description of the closures
of such orbits as sets of flags satisfying certain incidence conditions encoded by the parametrizing clans.
\end{abstract}

\maketitle

\section{Introduction}
Let $G$ be a reductive algebraic group and $B$ a Borel subgroup.  Let $\theta$ be an involution of $G$ (i.e. an 
automorphism with $\theta^2 = id$), and let $K=G^{\theta}$ be the subgroup of $G$ consisting of elements fixed by 
$\theta$.  $K$ is called a \textit{symmetric subgroup}.  $K$ acts on the flag variety $G/B$ with finitely many orbits,
and the geometry of these orbits and their closures is important in the representation theory of a certain real Lie
group, an associated ``real form'' of $G$.

In \cite{Richardson-Springer-90}, Richardson and Springer considered a certain partial order on the set of $K$-orbits, defined by $\caO_1 \leq \caO_2$ if and only if $\overline{\caO_1} \subseteq \overline{\caO_2}$.  They called this order the ``Bruhat order'', by analogy with the case of the Bruhat order on Schubert varieties.

In some specific cases, this Bruhat order is understood in a very explicit way.  For example, the set of $O_n$-orbits
(resp. $Sp_{2n}$-orbits) on $GL_n/B$ (resp. $GL_{2n}/B$) is parametrized by the set of involutions in $S_n$ (resp. the set
of fixed point-free involutions in $S_{2n}$), and the Bruhat order in each case is just the restriction of the ordinary
Bruhat order on the corresponding symmetric group.  This makes it easy to compare two orbits directly, without resorting to recursive computation, since the Bruhat order on $S_n$ is well known to be given by
\begin{equation}\label{eq:bruhat}
u \leq v \Leftrightarrow r_u(i,j) \geq r_v(i,j) \text{ for all } i,j,
\end{equation}
where $r_u(i,j) := \#\{k \leq i \mid u(k) \geq j\}$.

In addition to giving a very explicit understanding of the poset of $K$-orbits, this understanding of Bruhat order also
makes it straightforward to describe the orbit closures as sets of flags.  In the two cases above, such a description is
given in \cite{Wyser-13-TG}.  In the case of type $A$ Schubert varieties, \eqref{eq:bruhat} implies that the Schubert
variety $X^w = \overline{BwB/B}$ is described set-theoretically as follows:
\begin{equation}\label{eq:schubert-var}
X^w = \{F_{\bullet} \in GL_n/B \mid \dim(F_i \cap E_j) \geq r_w(i,j) \ \forall i,j\},
\end{equation}
where $E_{\bullet}$ denotes a fixed flag.

For symmetric pairs $(G,K)$ other than those mentioned above, the Bruhat order on $K$-orbits is not understood quite as
explicitly.  Using results of \cite{Richardson-Springer-90} together with explicit descriptions of the ``weak order'' (a
different partial order on the orbits, analogous to the weak Bruhat order on Schubert varieties, which is weaker than the
full Bruhat order) in \cite{Matsuki-Oshima-90}, one \textit{can} compute the Bruhat order in many examples.  However, the computation is recursive in nature, and does not obviously imply a description of the Bruhat order as explicit as \eqref{eq:bruhat}.  Consequently, in such cases we do not necessarily know a set-theoretic description of the $K$-orbit closures as explicit as \eqref{eq:schubert-var}.

The purpose of this paper is to give an explicit description of Bruhat order, similar in nature to \eqref{eq:bruhat}, and a
corresponding description of $K$-orbit closures as sets of flags, along the lines of \eqref{eq:schubert-var}, in a specific case, namely that of the type $A$
symmetric pair $(G,K)=(GL_{p+q},GL_p \times GL_q)$.  Taking the
involution $\theta$ to be conjugation by the diagonal matrix with $p$ $1$'s followed by $q$ $-1$'s on the diagonal,
$K$ is realized in the obvious way as the subgroup of $GL_{p+q}$ consisting of a $p \times p$ upper-left block, a
$q \times q$ lower-right block, and $0$'s outside of these blocks.  The orbits in this case are parametrized by what
have commonly been called ``clans'' \cite{Matsuki-Oshima-90,Yamamoto-97}.

\begin{definition}
A \textbf{clan of signature $(p,q)$} is an involution in $S_{p+q}$ with each fixed point decorated by either a $+$ or
a $-$ sign, in such a way that the number of $+$ fixed points minus the number of $-$ fixed points is $p-q$.  (If
$p < q$, then there should be $q-p$ more $-$ signs than $+$ signs.)

A clan is depicted by a string $c_1 \hdots c_n$ of $p+q$ characters.  Where the underlying involution fixes
$i$, $c_i$ is either a $+$ or a $-$, as appropriate.  Where the involution interchanges $i$ and $j$, 
$c_i=c_j\in \N$ is a matching pair of natural numbers.  (A different natural number is used for each 
pair of indices exchanged by the involution.)
\qed
\end{definition}

For example, suppose $n = 3$, and $p = 2, q = 1$.  Then we must consider all clans of length $3$ with one more $+$ than $-$ (since $p-q = 1$).  There are $6$ such.  They are displayed in Figure 1, which depicts the Bruhat order.  The minimal (closed) orbits are those at the bottom of the graph; their clans consist only of signs.  The open, dense orbit is the unique maximal orbit in the graph.  The larger case of $n=4$ and $p=q=2$ is given in Figure 2.

\begin{figure}[h!]
	\caption{$(GL(3,\C),GL(2,\C) \times GL(1,\C))$}\label{fig:type-a-2-1}
	\centering
	\includegraphics[scale=0.5]{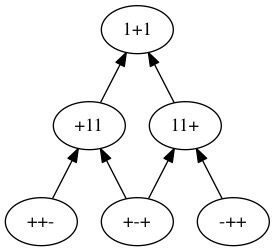}
\end{figure}

\begin{figure}[h!]
	\caption{$(GL(4,\C),GL(2,\C) \times GL(2,\C))$}\label{fig:type-a-2-2}
	\centering
	\includegraphics[scale=0.5]{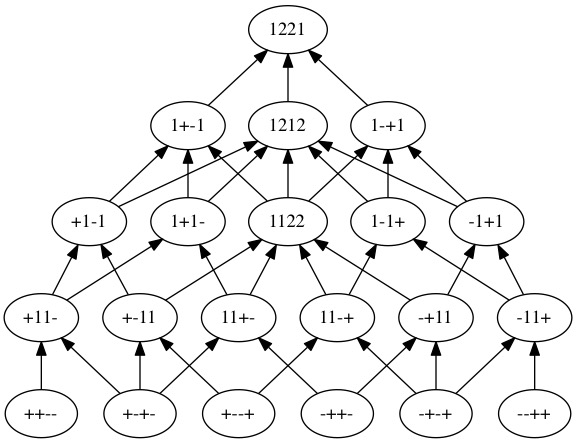}
\end{figure}

To state our results precisely, we define the following notations associated to any clan $\gamma=c_1 \hdots, c_n$,
and to any $i=1,\hdots,n$ or $i,j$ with $1 \leq i < j \leq n$.
\begin{enumerate}
	\item $\gamma(i; +) = $ the total number of plus signs and pairs of equal natural numbers occurring among $c_1 \hdots c_i$;
	\item $\gamma(i; -) = $ the total number of minus signs and pairs of equal natural numbers occurring among $c_1 \hdots c_i$; and
	\item $\gamma(i; j) = $ the number of pairs of equal natural numbers $c_s = c_t \in \N$ with $s \leq i < j < t$.
\end{enumerate}

As examples, for the $(2,2)$-clan $\gamma=1+1-$,
\begin{enumerate}
	\item $\gamma(i; +) = 0,1,2,2$ for $i=1,2,3,4$;
	\item $\gamma(i; -) = 0,0,1,2$ for $i=1,2,3,4$; and
	\item $\gamma(i;j) = 1,0,0,0,0,0$ for $(i,j) = (1,2), (1,3), (1,4), (2,3), (2,4), (3,4)$.
\end{enumerate}

\begin{theorem}\label{thm:bruhat}
 Let $\gamma,\tau$ be $(p,q)$-clans, and let $Y_{\gamma},Y_{\tau}$ be the corresponding $K$-orbit closures.
Then $\gamma \leq \tau$ (meaning $Y_{\gamma} \subseteq Y_{\tau}$) if and only if the following three inequalities
hold for all $i,j$:
\begin{enumerate}
 \item $\gamma(i;+) \geq \tau(i;+)$;
 \item $\gamma(i;-) \geq \tau(i;-)$;
 \item $\gamma(i;j) \leq \tau(i;j)$.
\end{enumerate}
\end{theorem}

Now define the following notations.  Let $E_p = \C \cdot \left\langle e_1,\hdots,e_p \right\rangle$ be the span of
the first $p$ standard basis vectors, and let
$\widetilde{E_q} = \C \cdot \left\langle e_{p+1},\hdots,e_n \right\rangle$ be the span of the last $q$ standard
basis vectors.  Let $\pi: \C^n \rightarrow E_p$ be the projection onto $E_p$.

Then from Theorem \ref{thm:bruhat}, we can deduce the following explicit set-theoretic description of $K$-orbit closures.

\begin{corollary}\label{cor:orbit-closure}
With notation as above, $Y_{\gamma}$ is precisely the set of all flags $F_{\bullet}$ satisfying the following three conditions
for all $i,j$:
\begin{enumerate}
 \item $\dim(F_i \cap E_p) \geq \gamma(i;+)$; 
 \item $\dim(F_i \cap \widetilde{E_q}) \geq \gamma(i;-)$; and
 \item $\dim(\pi(F_i) + F_j) \leq j + \gamma(i;j)$.
\end{enumerate}
\end{corollary}

\begin{remark}\label{rmk:smirnov}
As mentioned in \cite[Remark 3.4]{Wyser-15-GD}, a special case of Theorem \ref{thm:bruhat} follows from the results of the Ph.D.
thesis of E. Smirnov \cite[Theorem 3.10]{Smirnov-Thesis}.  Translated to our combinatorial parameters, Smirnov's result 
implies that if $\gamma(i;\pm) = \tau(i;\pm)$ for all $i$, then $\gamma \leq \tau$ if and only if $\gamma(i;j) \leq 
\tau(i;j)$ for all $i,j$.
\qed
\end{remark}

\begin{remark}\label{rmk:richardson-springer-conj}
Theorem \ref{thm:bruhat} confirms a special case of \cite[Conjecture 5.6.2]{Richardson-Springer-92}.
\qed
\end{remark}

\begin{remark}
Corollary \ref{cor:orbit-closure} is useful because it is a first step in studying these $K$-orbit closures, and affine open
subsets thereof, via local equations.  Such study allows us to analyze singularities, important in representation theory,
and also allows us to understand cohomological invariants in combinatorial ways, via Gr\"{o}bner degeneration.

Indeed, despite not having yet appeared in published form, Corollary \ref{cor:orbit-closure} has already been used in such
ways in the
papers \cite{Wyser-Yong-13, Woo-Wyser-14}.  It will also be used in an article in preparation \cite{Woo-Wyser-Yong-15+} to
establish a local isomorphism between ``Mars-Springer varieties'', which are certain ``attractive slices'' of orbit
closures.

Additionally, Corollary \ref{cor:orbit-closure} is used in \cite{Wyser-15-GD} to explicitly describe certain
types of degeneracy loci parametrized by the $K$-orbit closures in this case.
\qed
\end{remark}

\begin{remark}
Theorem 1.2 makes it possible to study topological properties of the poset of clans.  In cases where the Bruhat order on
$K$-orbits is understood explicitly, and where such study has been undertaken, the posets have been shown to have favorable
combinatorial properties.  We briefly recall some of these results.

By results of Richardson-Springer \cite{Richardson-Springer-90}, any Bruhat poset of $K$-orbits is \textit{pure}, meaning
that all maximal chains in such a poset have the same length.

The set $\caI$ of involutions of the symmetric group is known to be \textit{bounded} (meaning it has a unique minimal element and
a unique maximal element), hence it is \textit{ranked} (pure and bounded).  In \cite{Incitti-04}, the rank function is
computed explicitly, and $\caI$ is additionally shown to be \textit{Eulerian} (all intervals of length at least $1$ 
have the same number of elements of odd and even rank) and \textit{EL shellable}.  (The reader
may consult e.g. \cite{Stanley-EC} for the definition of this last term.)

The set $\caI_{fpf}$ of fixed point-free involutions of the symmetric group is also known to be bounded, hence ranked.
This poset is shown in \cite{CCT-15} to also be EL shellable.

Any finite Weyl group (indexing Schubert varieties in a flag variety, which can be viewed in a certain sense
as a special case of $K$-orbits), indeed, any finite Coxeter group, is known to have a number of nice properties:  all
are bounded, ranked, \textit{thin} (meaning that each interval of length $2$ is a $4$-element ``diamond'') \cite{Bjorner-Wachs}, Eulerian \cite{Verma-71,Deodhar-77}, and
EL shellable \cite{Dyer-93}.

By contrast with some of the examples above, note that the poset of clans is not bounded, as it has numerous minimal
elements.  Even if one artifically bounds the poset by formally adjoining a minimum element, the resulting poset is ranked,
but neither
thin nor Eulerian in general, as one can see by examining Figures 1 and 2.  For instance, in Figure 1, the interval $[++-,1+1]$ is a linear
chain, while in Figure 2, the length $2$ interval $[1122,1221]$ consists of $5$ elements.

I do not know whether this poset is shellable.
\qed
\end{remark}

\begin{question}\label{q:shellable}
 Is the order complex of the Bruhat poset of clans shellable?  Is the poset CL shellable?  EL shellable?
 \qed
\end{question}

After recalling some known results in Section \ref{sec:yamamoto-stuff}, we give the proof of Theorem \ref{thm:bruhat}
and Corollary \ref{cor:orbit-closure} in Section \ref{sec:proof-of-theorem}.

\section{Describing the Bruhat order on clans}
In this section, we give the proofs of Theorem \ref{thm:bruhat} and Corollary \ref{cor:orbit-closure}.  Before doing so, we quickly recall some background on
the parametrization and the set-theoretic descriptions of the orbits themselves (\textit{not} their closures) which will
be used in the proof.

\subsection{Known results on $K$-orbits}\label{sec:yamamoto-stuff}
Our reference for the results of this section is \cite{Yamamoto-97}.

For a clan $\gamma$, recall our definitions of the numbers $\gamma(i;\pm)$ and $\gamma(i;j)$.  (We will call these
numbers the ``rank numbers of $\gamma$'' for short.)
\begin{theorem}[\cite{Yamamoto-97}]\label{thm:orbit_description}
Suppose $p+q=n$.  For a $(p,q)$-clan $\gamma$, define $Q_{\gamma}$ to be the set of all flags $F_{\bullet}$ having the following three properties for all $i,j$ ($i<j$):
\begin{enumerate}
	\item $\dim(F_i \cap E_p) = \gamma(i; +)$
	\item $\dim(F_i \cap \widetilde{E_q}) = \gamma(i; -)$
	\item $\dim(\pi(F_i) + F_j) = j + \gamma(i; j)$
\end{enumerate}

For each $(p,q)$-clan $\gamma$, $Q_{\gamma}$ is nonempty, stable under $K$, and in fact is a single $K$-orbit on $G/B$.

Conversely, every $K$-orbit on $G/B$ is of the form $Q_{\gamma}$ for some $(p,q)$-clan $\gamma$.  Hence the
association $\gamma \mapsto Q_{\gamma}$ defines a bijection between the set of all $(p,q)$-clans and the set of
$K$-orbits on $G/B$.
\end{theorem}
Note that in light of Theorem \ref{thm:orbit_description}, Corollary \ref{cor:orbit-closure} says that we pass from an
orbit to its closure by changing the equalities in the set-theoretic
description of the orbit to inequalities.
This is the case when one passes from a Schubert cell to the corresponding Schubert variety as well.

We next outline an algorithm, described in \cite{Yamamoto-97}, for producing a representative of $Q_{\gamma}$ given
the clan $\gamma$.

First, for each pair of matching natural numbers of $\gamma$, assign one of the numbers a ``signature" of $+$, and the
other a signature of $-$.  Now choose a permutation $\sigma$ of $1,\hdots,n$ with the following properties for all
$i=1,\hdots,n$:

\begin{enumerate}
	\item $1 \leq \sigma(i) \leq p$ if $c_i = +$ or if $c_i \in \N$ and the signature of $c_i$ is $+$.
	\item $p+1 \leq \sigma(i) \leq n$ if $c_i = -$ or if $c_i \in \N$ and the signature of $c_i$ is $-$.
\end{enumerate}

Having determined such a permutation $\sigma$, take $F_{\bullet} = \left\langle v_1, \hdots, v_n \right\rangle$ to be
the flag specified as follows:
\[ v_i = 
\begin{cases}
	e_{\sigma(i)} & \text{ if $c_i=\pm$}, \\
	e_{\sigma(i)} + e_{\sigma(j)} & \text{ if $c_i \in \N$, $c_i$ has signature $+$, and $c_i = c_j$}, \\
	-e_{\sigma(i)} + e_{\sigma(j)} & \text{ if $c_i \in \N$, $c_i$ has signature $-$, and $c_i = c_j$.}
\end{cases} \]

For example, for the orbit corresponding to the clan $+++---$, we could take $\sigma = 1$, which would give the
standard flag $\left\langle e_1,\hdots,e_6 \right\rangle$.  For $1-+1$, we could assign signatures to the $1$'s as
follows:  $1_+-+1_-$.  We could then take $\sigma$ to be the permutation $1324$.  This would give the flag 
	\[ F_{\bullet} = \left\langle e_1 + e_4, e_3, e_2, e_1-e_4 \right\rangle. \]

\subsection{The proofs of Theorem \ref{thm:bruhat} and Corollary \ref{cor:orbit-closure}}\label{sec:proof-of-theorem}
Our proof of Theorem \ref{thm:bruhat} is along the same lines as the proof of
\cite[\S10.5, Proposition 7]{Fulton-YoungTableaux}, regarding closures of type $A$ Schubert cells and the Bruhat order
on $S_n$.  We proceed as follows:

\begin{enumerate}
	\item Define a ``combinatorial Bruhat order" $\leq$ on $(p,q)$-clans, which we secretly know reflects the true geometric Bruhat order.
	\item Describe the covering relations with respect to this combinatorial Bruhat order.
	\item Show that orbit closures $Y_{\gamma}$, $Y_{\tau}$ satisfy $Y_{\gamma} \subseteq Y_{\tau}$ if and only if $\gamma \leq \tau$.
\end{enumerate}

One direction of (3) is easy.  The other requires knowledge of the covering relations with respect to the order $\leq$,
acquired in (2) above, together with some limiting arguments involving the specific representatives of the orbits
$Q_{\gamma}$ and $Q_{\tau}$ (described in Section \ref{sec:yamamoto-stuff}) when $\tau$ covers $\gamma$ in the
combinatorial order.

To see that the combinatorial order that we want to define is in fact a partial order, we first observe the following
easy fact.

\begin{prop}\label{prop:rank-nums-determine-clan}
The rank numbers $\gamma(i;+)$, $\gamma(i;-)$, and $\gamma(i;j)$ determine $\gamma$ uniquely.
\end{prop}
\begin{proof}
Say that $\gamma=c_1 \hdots c_n$ ``has a $+$ jump at $i$" if $\gamma(i;+) = \gamma(i-1;+) + 1$.  Likewise, say that
$\gamma$ ``has a $-$ jump at $i$" if $\gamma(i;-) = \gamma(i-1;-) + 1$.  From the definitions, it is clear that
$\gamma$ has a $+$ jump at $i$ if and only if $c_i$ is a $+$ or the second occurrence of a natural number, and does
\textit{not} have a $+$ jump at $i$ if and only if $c_i$ is a $-$ or the first occurrence of a natural number.
Likewise, $\gamma$ has a $-$ jump at $i$ if and only if $c_i$ is a $-$ or the second occurrence of a natural number,
and does \textit{not} have a $-$ jump at $i$ if and only if $c_i$ is a $+$ or the first occurrence of a natural
number.

Thus the numbers $\gamma(i;+)$ and $\gamma(i;-)$, by themselves, completely determine the location of $+$'s, $-$'s,
first occurrences of natural numbers, and second occurrences of natural numbers.  The only choice left in constructing
$\gamma$, then, is when we see the second occurrence of a natural number, \textit{which} natural number is it the
second occurrence \textit{of}?  This is determined by the numbers $\gamma(i;j)$.  Let $k$ be the first index at which
$\gamma$ has the second occurrence of a natural number.  Supposing there is only one first occurrence to the left of
position $k$, $c_k$ is determined.  So suppose there is more than one first occurrence to the left of position $k$,
with $i_1,\hdots,i_m$ the indices less than $k$ at which $\gamma$ has first occurrences.  Consider the numbers
$\gamma(i_1; k), \gamma(i_2; k), \hdots, \gamma(i_m; k)$.  From the definitions, it is clear that for all
$l=1,\hdots,m$, $\gamma(i_l; k) \leq l$, and also that $\gamma(i_m; k) = m-1$.  Thus there is some first $l$ at which
$\gamma(i_l;k) < l$.  Then $c_k$ must be the second occurrence of $c_{i_l}$.  Indeed, $c_k$ cannot be the second
occurrence of any $c_{i_j}$ with $j < l$, because if it were, the pair $(c_{i_j},c_k)$ would not be counted in the
number $\gamma(i_j; k)$, so we would necessarily have $\gamma(i_j; k) < j$.  On the other hand, if the second
occurrence of $c_{i_l}$ were beyond position $k$, then we would necessarily have that $\gamma(i_l; k) = l$.  Thus
$c_k$ is determined by the numbers $\gamma(i_1; k), \hdots, \gamma(i_m; k)$.

Working in order from left to right, the remaining indices at which $\gamma$ has second occurrences can be filled in
using similar logic.  At each such index, we consider the indices left of $k$ at which $\gamma$ has the first
occurrence of a natural number \textit{which does not yet have a mate}.  Applying the same argument as above will
allow us to determine which of those first occurrences $c_k$ should be the mate of.
\end{proof}

\begin{example}
As an example of the preceding proof, suppose that we are told that the $(3,3)$ clan $\gamma$ has rank numbers
\begin{itemize}
	\item $\gamma(i;+) = 0,0,1,2,2,3$ for $i=1,2,3,4,5,6$;
	\item $\gamma(i;-) = 0,0,1,2,2,3$ for $i=1,2,3,4,5,6$;
	\item $\gamma(i;j) = 1,1,0,0,0,1,0,0,0,0,0,0,0,0,0$ for $(i,j) = (1,2),(1,3),(1,4),(1,5),(1,6),\newline (2,3),(2,4),(2,5),(2,6),(3,4),(3,5),(3,6),(4,5),(4,6),(5,6)$.
\end{itemize}

The first two sets of numbers tell us that $\gamma$ has the pattern $FFSSFS$, where $F$ represents the first
occurrence of a natural number, and $S$ represents the second occurrence of a natural number.  We can assign natural
numbers to the positions of the $F$'s any way we'd like --- they may as well be $1,2,3$, in order.  Thus $\gamma$ has
the form $12SS3S$.  Looking at the first $S$ in position $3$, we need to determine whether $c_3 = 2$ or $c_3 = 1$.
Looking at the numbers $\gamma(1;3) = 1$ and $\gamma(2;3) = 1$, we see that $c_3 = 2$.  This forces $c_4 = 1$, and
$c_6 = 3$.  Thus $\gamma = 122133$.
\end{example}

We now define combinatorial Bruhat order in the obvious way.
\begin{definition}
We define the \textbf{combinatorial Bruhat order} $\leq$ on $(p,q)$-clans as follows:  $\gamma \leq \tau$ if and only 
\begin{enumerate}
	\item $\gamma(i; +) \geq \tau(i;+)$ for all $i$;
	\item $\gamma(i; -) \geq \tau(i;-)$ for all $i$;
	\item $\gamma(i; j) \leq \tau(i; j)$ for all $i<j$.
\end{enumerate}
\end{definition}

In light of Proposition \ref{prop:rank-nums-determine-clan}, it is clear that $\leq$ is a partial order on
$(p,q)$-clans.  Our first task is to describe the covering relations with respect to this order.  Recall that clans can
be thought of as involutions in $S_n$ with signed fixed points.  Conversely, any involution in $S_n$ can likewise be
encoded by a character string, as we do for clans, but with the character string consisting of pairs of
matching natural numbers and (say) dots, rather than signs, with the dots marking fixed points.  (So, for example,
the involution $563412 = (1,5)(2,6)$ corresponds to the character string
$12 \cdot \cdot 12$.)  We first observe that when involutions of $S_n$ are translated into such
``clan-like symbols,'' their Bruhat order (by which we mean the restriction of the Bruhat order on $S_n$ to the subset
of involutions) is closely related to the combinatorial Bruhat order on clans defined above.

Paralleling the notation for clans, given an involution $\gamma \in S_n$, with $\gamma_1 \hdots \gamma_n$ the
corresponding character string just described, define the following for any $i =1,\hdots,n$ and for any
$1 \leq s < t \leq n$:
\begin{enumerate}
	\item $\gamma(i; \cdot) = $ Number of dots plus twice the number of pairs of equal natural numbers \newline
occurring among $\gamma_1 \hdots \gamma_i$; and
	\item $\gamma(s;t) = \#\{\gamma_a = \gamma_b \in \N \mid a \leq s, b > t\}$.
\end{enumerate}

We omit the easy proof of the following proposition, which is simply a reformulation of the Bruhat order on
involutions which more closely mirrors our definition of the combinatorial Bruhat order on clans.
\begin{prop}\label{prop:bruhat-involutions}
For involutions $\gamma,\tau \in S_n$, $\gamma \leq \tau$ in Bruhat order if and only if
\begin{enumerate}
	\item $\gamma(i; \cdot) \geq \tau(i; \cdot)$ for $i=1,\hdots,n$; and
	\item $\gamma(s;t) \leq \tau(s;t)$ for all $1 \leq s < t \leq n$.
\end{enumerate}
\end{prop}

\begin{corollary}\label{cor:bruhat-involutions}
Suppose $\gamma$ and $\tau$ are two $(p,q)$-clans.  Then $\gamma \leq \tau$ in the combinatorial Bruhat order if and
only if
\begin{enumerate}
	\item The underlying involutions of $\gamma$, $\tau$ are related in the Bruhat order on $S_n$, and additionally
	\item $\gamma(i; +) \geq \tau(i;+)$ and $\gamma(i;-) \geq \tau(i;-)$ for all $i$.
\end{enumerate}
\end{corollary}
\begin{proof}
Clearly, $\gamma(i;+) + \gamma(i;-) = \gamma(i;\cdot)$, and $\tau(i;+) + \tau(i;-) = \tau(i; \cdot)$.  So if
$\gamma \leq \tau$ in the combinatorial Bruhat order, then $\gamma(i;+) \geq \tau(i;+)$ and $\gamma(i;-) \geq
\tau(i;-)$, so $\gamma(i; \cdot) \geq \tau(i; \cdot)$.  Thus the underlying involutions are related in the Bruhat
order by Proposition \ref{prop:bruhat-involutions}.

On the other hand, relation of the underlying involutions in the Bruhat order implies that $\gamma(i;j) \leq \tau(i;j)$
for all $i<j$ by Proposition \ref{prop:bruhat-involutions}, so insisting on the conditions $\gamma(i;+) \geq \tau(i;+)$
and $\gamma(i;-) \geq \tau(i;-)$ guarantees that $\gamma \leq \tau$ in the combinatorial order as well.
\end{proof}

We now characterize the covering relations with respect to the combinatorial Bruhat order, using Corollary
\ref{cor:bruhat-involutions} to simplify some of the arguments.  Indeed, by use of Corollary \ref{cor:bruhat-involutions},
in certain instances we are able to use the results of \cite{Incitti-04} regarding the covering relations in the
Bruhat order on ordinary involutions of $S_n$.

\begin{theorem}\label{thm:covers-combinatorial-order}
	Suppose that $\gamma$, $\tau$ are $(p,q)$-clans, with $\gamma < \tau$.  Then there exists a clan $\gamma'$ such that $\gamma < \gamma' \leq \tau$, and such that $\gamma'$ is obtained from $\gamma$ by a ``move" of one of the following types:
	\begin{enumerate}
		\item Replace a pattern of the form $+-$ by $11$, i.e. replace a plus and minus by a pair of matching natural numbers.
		\item Replace a pattern of the form $-+$ by $11$.
		\item Replace a pattern of the form $11+$ by $1+1$, i.e. interchange a number and a $+$ sign to its right if in doing so you move the number farther from its mate.
		\item Replace a pattern of the form $11-$ by $1-1$.
		\item Replace a pattern of the form $+11$ by $1+1$.
		\item Replace a pattern of the form $-11$ by $1-1$.
		\item Replace a pattern of the form $1122$ by $1212$.
		\item Replace a pattern of the form $1122$ by $1+-1$.
		\item Replace a pattern of the form $1122$ by $1-+1$.
		\item Replace a pattern of the form $1212$ by $1221$.
	\end{enumerate}
\end{theorem}
\begin{proof}
We prove this by explicitly constructing $\gamma'$ from $\gamma$ and $\tau$, then showing that the $\gamma'$ so
constructed satisfies $\gamma < \gamma' \leq \tau$.  (This is the clan version of what is done in \cite[\S 10.5,
Lemma 11]{Fulton-YoungTableaux} in the case of permutations.)

In fact, when one translates the clans $\gamma$, $\gamma'$, and $\tau$ to their underlying involutions, one sees that
the construction of $\gamma'$ follows very closely the ``covering moves" for involutions described in
\cite[\S 3-4]{Incitti-04}.  In fact, we are able to apply directly the results of \cite{Incitti-04} regarding these
covering moves in certain cases.

The proof is by case analysis.  As the statement of the theorem may indicate, there are a number of cases to consider.
First, we say that clans $\gamma=\gamma_1 \hdots \gamma_n$ and $\tau=\tau_1 \hdots \tau_n$ differ at position $i$ if
and only if one of the following holds:
\begin{itemize}
	\item $\gamma_i$ and $\tau_i$ are opposite signs;
	\item $\gamma_i$ is a sign and $\tau_i$ is a number;
	\item $\gamma_i$ is a number and $\tau_i$ is a sign; or
	\item $\gamma_i$ and $\tau_i$ are numbers whose mates are in different positions. 
\end{itemize}

If $\gamma < \tau$, then obviously there is a first position at which $\gamma$ and $\tau$ differ.  Denote this
position by $f$.  One of the following must hold:
\begin{enumerate}
	\item $\gamma_f$ is a $+$ sign, and $\tau_f$ is the first occurrence of some natural number;
	\item $\gamma_f$ is a $-$ sign, and $\tau_f$ is the first occurrence of some natural number; or
	\item $\gamma_f$ and $\tau_f$ are both first occurrences of a natural number, with $\gamma_f = \gamma_i$ and
$\tau_f = \tau_j$ for $f < i < j$.
\end{enumerate}

To see this, note first that neither $\gamma_f$ nor $\tau_f$ can be the second occurrence of a natural number, since
if either was, a difference would've been detected prior to position $f$, namely at the position of the first
occurrence of that natural number.  Thus the only possibilities for $(\gamma_f,\tau_f)$ are $(+,-)$, $(-,+)$, $(+,F)$,
$(-,F)$, $(F,+)$, $(F,-)$, $(F,F)$, where here `F' stands for the first occurrence of a natural number.  We can rule
out the cases $(+,-)$, $(-,+)$, $(F,+)$, and $(F,-)$, since in those cases we would not have $\gamma < \tau$ in the
combinatorial Bruhat order.  Indeed,
\begin{itemize}
	\item $(\gamma_f,\tau_f) = (+,-) \Rightarrow \gamma(f; -) < \tau(f; -)$;
	\item $(\gamma_f,\tau_f) = (-,+) \Rightarrow \gamma(f; +) < \tau(f; +)$;
	\item $(\gamma_f,\tau_f) = (F,+) \Rightarrow \gamma(f; +) < \tau(f; +)$;
	\item $(\gamma_f,\tau_f) = (F,-) \Rightarrow\gamma(f; -) < \tau(f; -)$.
\end{itemize}

Note also that in the case $(F,F)$, if we had $\gamma_f = \gamma_i$ and $\tau_f = \tau_j$ with $j < i$, then this
would imply that $\gamma(f;j) > \tau(f;j)$, which is again contrary to our assumption that $\gamma < \tau$.  Thus in
the case $(F,F)$, the mate for $\tau_f$ must occur to the \textit{right} of the mate for $\gamma_f$.

We remark that the position we are denoting $f$ (for ``first") is the same as the ``difference index" $di$ defined in
\cite[Definition 4.1]{Incitti-04}.  In fact, there is no loss in generality in assuming that $f=1$, and we do so in
what follows.  This essentially allows us to ignore, over certain ranges of indices, any second occurrences of
natural numbers whose first occurrences are prior to position $f$.  Since $\gamma$ and $\tau$ match in all such
positions anyway, considering them does nothing but clutter our arguments unnecessarily.  So from this point forward,
we assume $f=1$.

We now consider the cases $(\gamma_1,\tau_1) = (+,F)$, $(-,F)$, and $(F,F)$ in turn.

\begin{case}[$(\gamma_1,\tau_1) = (+,F)$]
Suppose that $\gamma_1 = +$, and that $\tau_1$ is the first occurrence of a natural number whose second occurrence is
at position $i$:  $\tau_1 = \tau_i \in \mathbb{N}$.  We claim first that there exists $j \in [1,i]$ such that either
$\gamma_j$ is a $-$ sign, or $\gamma_j$ is the first of a pair of matching natural numbers whose second occurence is
in position at most $i$, i.e. $\gamma_j = \gamma_k \in \N$ $1 < j < k \leq i$.  To see this, note that $\tau(1;-) =
\gamma(1;-) = 0$, and that $\tau$ has a $-$ jump at position $i$, $\tau_i$ being the second occurrence of $\tau_1$.
Thus in order to ensure $\gamma(i;-) \geq \tau(i;-)$, $\gamma$ must have a $-$ jump in the range $[1,i]$.  This can
occur only at either a $-$ sign, or the second occurrence of a natural number.

Let $j \in [1,i]$ be the smallest index such that $\gamma_j$ is either a $-$ or the first occurrence of a natural number whose second occurrence is in position at most $i$, with $\gamma_j = \gamma_k \in \N$ for $1 < j < k \leq i$.  We consider the two subcases separately, but first we prove the following lemma, which is of use in both instances.

\begin{lemma}\label{lem:1.A}
For any $l \in [1,j-1]$, let $\gamma_F(l)$ denote the number of first occurrences of natural numbers among $\gamma_1 \hdots \gamma_l$, and define $\tau_F(l)$ similarly.  Then $\gamma_F(l) < \tau_F(l)$.
\end{lemma}
\begin{proof}
Note that, by our choice of $j$, all first occurrences for $\gamma$ in the range $[1,l]$ have their second occurrence strictly after position $i$.  Since $\gamma(1;i) \leq \tau(1;i)$, $\tau$ must have at least as many first occurrences in this range whose second occurrences are strictly after position $i$.  And $\tau$ has (at least) one more first occurrence, namely $\tau_1$, whose second occurrence is \textit{at} position $i$.
\end{proof}

\begin{subcase}[$\gamma_j = -$]
This case looks like the following:
\[
	\begin{array}{rccccc}
		&&&&i& \\
		\tau: & 1 & \hdots & \hdots & 1 & \hdots \\
		\gamma: & + & \hdots & - & \hdots & \hdots \\
		&&&j&&
	\end{array}
\]

In this case, we claim that the clan $\gamma'$ obtained from $\gamma$ by replacing $\gamma_1$ and $\gamma_j$ by a new pair of matching natural numbers (i.e. replacing the pattern $(\gamma_f,\gamma_j) = +-$ by $11$) satisfies $\gamma < \gamma' \leq \tau$.  To see that $\gamma < \gamma'$, we simply note how the rank numbers for $\gamma'$ differ from those of $\gamma$.  From the definitions, it is clear that the only changes in the rank numbers are
\begin{itemize}
	\item $\gamma'(k;+) = \gamma(k;+) - 1$ for all $k \in [1,j-1]$;
	\item $\gamma'(k;l) = \gamma(k;l) + 1$ for all $k<l$ with $1 \leq k < l < j$.
\end{itemize}

Thus $\gamma < \gamma'$.  To see that $\gamma' \leq \tau$, having noted how the rank numbers for $\gamma'$ differ from those of $\gamma$, and knowing that $\gamma < \tau$, we need only establish the following:
\begin{itemize}
	\item $\tau(k;+) < \gamma(k;+)$ for all $k \in [1,j-1]$; and
	\item $\tau(k;l) > \gamma(k;l)$ for all $k<l$ with $1 \leq k < l < j$.
\end{itemize}

We start with the first statement.  Since $\gamma(1;+) = 1$ and $\tau(1;+) = 0$, it is clear for $k=1$, and so we consider $k \in [2,j-1]$.  Consider the following observations:
\begin{enumerate}
	\item By our choice of $j$, $\gamma$ has no $-$ jumps in the range $[2,j-1]$.
	\item $\gamma(1;-) = \tau(1;-) = 0$.  Thus, by the previous bullet, and because $\gamma < \tau$, we must have that $\tau(k;-) = \gamma(k;-)$ for all $k \in [2,j-1]$.  As a consequence, note that there can be no $-$ jumps for $\tau$ in the range $[2,j-1]$, i.e. there can be no $k \in [2,j-1]$ where $\tau_k = -$, or where $\tau_k$ is a second occurrence of a natural number.
	\item Thus in the range $[2,j-1]$, the only possible characters for both $\tau$ and $\gamma$ are $+$'s and first occurrences of natural numbers.
\end{enumerate}

Since the number of first occurrences for $\gamma$ in the range $[1,k]$ is strictly less than the number of first occurrences for $\tau$ in the same range by Lemma \ref{lem:1.A}, the number of $+$ signs for $\gamma$ must be strictly greater than the number of $+$ signs for $\tau$ in this range.  Thus $\gamma(k;+) > \tau(k;+)$, as desired.

Now, we must see that $\tau(k;l) > \gamma(k;l)$ whenever $1 \leq k < l < j$.  Let such a $k,l$ be given.

By the above observations, it is clear that any first occurrence of a natural number for $\gamma$ occurring in the range $[1,k]$ has its mate in position strictly greater than $i$ (and hence strictly greater than $l$, since $l < j \leq i$), while any first occurrence of a natural number for $\tau$ occurring in this range has its mate in position at least $j$ (hence also in a position strictly greater than $l$).  Thus $\gamma(k;l)$ and $\tau(k;l)$ are simply the numbers of first occurrences for $\gamma$ and $\tau$, respectively, occurring in the range $[1,k]$.  So again it follows from Lemma \ref{lem:1.A} that $\gamma(k;l) < \tau(k;l)$.

This completes the proof that $\gamma < \gamma' \leq \tau$.
\end{subcase}

\begin{subcase}[$\gamma_j = \gamma_k \in \mathbb{N}$, $1 < j < k \leq i$]
Here, we have the following setup:
\[
	\begin{array}{rccccccc}
		&&&&&&i& \\
		\tau: & 1 & \hdots & \hdots & \hdots & \hdots & 1 & \hdots \\
		\gamma: & + & \hdots & 1 & \hdots & 1 & \hdots & \hdots \\
		&&&j&&k&
	\end{array}
\]

In this case, we claim that the $\gamma'$ obtained from $\gamma$ by interchanging $\gamma_1$ and $\gamma_j$ (replacing the substring $(\gamma_1,\gamma_j,\gamma_k) = +11$ by  $1+1$) satisfies $\gamma < \gamma' \leq \tau$.

To see that $\gamma < \gamma'$, simply note how the rank numbers for $\gamma'$ and $\gamma$ differ.  The only differences are
\begin{itemize}
	\item $\gamma'(l;+) = \gamma(l;+) - 1$ for $l=1,\hdots,j-1$; and 
	\item $\gamma'(l;m) = \gamma(l;m) + 1$ for pairs $l<m$ such that $1 \leq l < j$ and $l < m < k$.
\end{itemize}

This shows that $\gamma < \gamma'$.  To see that $\gamma' \leq \tau$, one shows that
\begin{itemize}
	\item $\tau(l;+) < \gamma(l;+)$ for $l=1,\hdots,j-1$; and
	\item $\tau(l;m) > \gamma(l;m)$ whenever $1 \leq l < j$ and $l<m<k$.
\end{itemize}

The proofs of these facts are exactly the same as those given in the previous subcase, so we do not repeat them here.
\end{subcase}
\end{case}

\begin{case}[$(\gamma_f,\tau_f) = (-,F)$]
In this case, we are able to make the exact same arguments as in the previous case, interchanging $-$ and $+$ signs everywhere.  The upshot is that $\gamma'$ can be obtained from $\gamma$ by a move of either type $-+ \rightarrow 11$ or $-11 \rightarrow 1-1$.
\end{case}

\begin{case}[$(\gamma_1,\tau_1) = (F,F)$]
Now, suppose that $(\gamma_1,\tau_1) = (F,F)$, with $\gamma_1$ and $\tau_1$ natural numbers such that $\gamma_1 = \gamma_j \in \mathbb{N}$ and $\tau_1 = \gamma_i$.  Recall that $j < i$.

There are multiple subcases to consider here.  To see what they are, we first note the following easy lemma.

\begin{lemma}\label{lem:3.B}
	One of the following must be true:
	\begin{enumerate}
		\item There exists a pair of matching numbers $\gamma_k = \gamma_l \in \N$ with $1 < k < j < l \leq i$.
		\item There exists a pair of natural numbers $\gamma_k = \gamma_l \in \N$ with $j < k < l \leq i$.
		\item There exists a $\gamma_k = \pm$ for some $k \in [j+1,i]$.
		\item There exists a pair of natural numbers $\gamma_k = \gamma_l \in \N$ with $l > i$, and with $k \in [j+1,i]$.
	\end{enumerate}
\end{lemma}
\begin{proof}
	Indeed, consider the possible values of $\gamma_{j+1},\hdots,\gamma_i$.  If any is a $\pm$, we are done, as we are in case (3).  If any is a second occurrence, then its first occurrence must occur either between $1$ and $j-1$ (case (1)), or else after $j$ (case (2)).  And if any is a first occurrence, then its second occurrence must occur either in position at most $i$ (case (2)), or in a position strictly beyond $i$ (case (4)).
\end{proof}

We consider each of the above cases in turn.

\begin{subcase}[There exist $\gamma_k = \gamma_l \in \N$ with $1 < k < j < l \leq i$.]\label{sc:1212}
The picture here is as follows:
\[
	\begin{array}{rccccccccc}
		&&&&&&&&i& \\
		\tau: & 1 & \hdots & \hdots & \hdots & \hdots & \hdots & \hdots & 1 & \hdots \\
		\gamma: & 1 & \hdots & 2 & \hdots & 1 & \hdots & 2 & \hdots & \hdots \\
		&&&k&&j&&l&&
	\end{array}
\]

Then choose the pair such that $k$ is minimal.  Note that the numbers $(\gamma_1,\gamma_k,\gamma_j,\gamma_l)$ form the pattern $1212$.  Let $\gamma'$ be obtained from $\gamma$ by changing this pattern to $1221$ (i.e. by either interchanging $\gamma_1$ and $\gamma_k$, or $\gamma_j$ and $\gamma_l$).  Then we claim that $\gamma < \gamma' \leq \tau$.

To see that $\gamma < \gamma'$, we simply note that the only changes in the rank numbers as we move from $\gamma$ to $\gamma'$ are that $\gamma'(s;t) = \gamma(s;t) + 1$ whenever $1 \leq s < k$ and $j \leq t < l$.  To see that $\gamma' \leq \tau$, then, we simply need to see that $\gamma(s;t) < \tau(s;t)$ for such $s$ and $t$.

In fact, by Corollary \ref{cor:bruhat-involutions}, it suffices here to observe that the underlying involutions of $\gamma'$ and $\tau$ are suitably related in Bruhat order.  That they are follows immediately from \cite[Corollary 4.6]{Incitti-04}.  Indeed, one checks that, on the underlying involutions, the move made here is the ``minimal covering transformation" of $\gamma$ relative to $\tau$, with $(1,k)$ being a ``non-crossing $ee$-rise".
\end{subcase}

\begin{subcase}[There is either a $+$, a $-$, or a pair of natural numbers occurring in the range $j+1,\hdots,i$.]\label{sc:1122}
In this case, choose $k$ to be the smallest index in the range $[j+1,i]$ such that either $\gamma_k = +$, $\gamma_k = -$, or $\gamma_k = \gamma_l \in \N$ with $k < l \leq i$.  We treat each of the three possibilities in turn, starting with the last.

\begin{subsubcase}[$\gamma_k = \gamma_l \in \N$ with $j < k < l \leq i$.]
We have the following picture:
\[
	\begin{array}{rccccccccc}
		& & & & & & & & i & \\
		\tau: & 1 & \hdots & \hdots & \hdots & \hdots & \hdots & \hdots & 1 & \hdots \\
		\gamma: & 1 & \hdots & 1 & \hdots & 2 & \hdots & 2 & \hdots & \hdots \\
		 & & & j & & k & & l & & 
	\end{array}
\]
	Note that the numbers $(\gamma_1,\gamma_j,\gamma_k,\gamma_l)$ form the pattern $1122$.  We claim that the $\gamma'$ obtained from $\gamma$ by changing this pattern either to $1+-1$ or $1-+1$ satisfies $\gamma < \gamma' \leq \tau$.  To see this, we first establish a few basic lemmas which will be used both in this subcase and the next.
	
	Recall the following notation, defined in a previous case:
	\[ \gamma_F(s) := \#\{t \in [1,s] \mid \gamma_t \text{ is a first occurrence}\}, \]
and
\[ \tau_F(s) := \#\{t \in [1,s] \mid \tau_t \text{ is a first occurrence}\}. \]
	\begin{lemma}\label{lem:3.2.A}
For any $s \in [1,k-1]$, 
\[ \gamma_F(s) \leq \tau_F(s). \]
\end{lemma}
\begin{proof}
Since we are not in Case \ref{sc:1212}, any first occurrence in the range $[2,j-1]$ has its mate in a position strictly beyond $i$.  By our choice of $k$, all first occurrences in the range $[j+1,k-1]$ also have their second occurrences in a position strictly beyond $i$.  Thus $\gamma_F(s) = \gamma(s;i) + 1$, the additional $1$ being $\gamma_1$, whose second occurrence is $\gamma_j$.  Likewise, $\tau(s;i) \leq \tau_F(s) + 1$, the additional $1$ being $\tau_1$, whose second occurrence is $\tau_i$.  Since $\gamma_F(s) + 1 = \gamma(s;i) \leq \tau(s;i) \leq \tau_F(s) + 1$, we get the desired inequality.
\end{proof}

\begin{lemma}\label{lem:3.2.B}
For any $s \in [j,k-1]$, if $\gamma(s;-) = \tau(s;-)$, then $\gamma(s;+) > \tau(s;+)$, and if $\gamma(s;+) = \tau(s;+)$, then $\gamma(s;-) > \gamma(s;-)$.
\end{lemma}
\begin{proof}
Since $\gamma(s;\pm) \geq \tau(s;\pm)$ by definition of the combinatorial Bruhat order, the statement here amounts to the fact that we cannot have both $\gamma(s;-) = \tau(s;-)$ and $\gamma(s;+) = \tau(s;+)$ at the same time at any point over this range of indices.

Consider the possible values of characters $\gamma_t$ for $t \in [2,s]$.  They are
\begin{itemize}
	\item $+$ signs;
	\item $-$ signs;
	\item First occurrences of natural numbers whose second occurrences are beyond position $s$;
	\item Pairs of numbers each occurring in the range $[2,s]$;
	\item The lone character $\gamma_j$ (the second occurrence of $\gamma_1$).
\end{itemize}

Consider the possible values of characters $\tau_t$ for $t \in [2,s]$.  They are
\begin{itemize}
	\item $+$ signs;
	\item $-$ signs;
	\item First occurrences of natural numbers whose second occurrences are beyond position $s$;
	\item Pairs of numbers each occurring in the range $[2,s]$.
\end{itemize}

We define the following notations:
\begin{itemize}
	\item $\gamma_+(s) = \#\{t \in [2,s] \mid \gamma_t = +\}$;
	\item $\gamma_-(s) = \#\{t \in [2,s] \mid \gamma_t = -\}$;
	\item $\gamma_F(s) = \#\{t \in [2,s] \mid \gamma_t = \gamma_r \in \N, r > s\}$;
	\item $\gamma_P(s) = \#\{ \text{Pairs of natural numbers occurring in the range $[2,s]$}\}$.
\end{itemize}

We use similar notations for $\tau$.

The following observations are evident:
\begin{itemize}
	\item $\gamma(s;+) = \gamma_+(s) + \gamma_P(s) + 1$;
	\item $\gamma(s;-) = \gamma_-(s) + \gamma_P(s) + 1$;
	\item $\tau(s;+) = \tau_+(s) + \tau_P(s)$;
	\item $\tau(s;-) = \tau_-(s) + \tau_P(s)$;
	\item $s - 1 = \gamma_+(s) + \gamma_-(s) + \gamma_F(s) + 2 \cdot \gamma_P(s) + 1$;
	\item $s - 1 = \tau_+(s) + \tau_-(s) + \tau_F(s) + 2 \cdot \tau_P(s)$.
\end{itemize}

(The various ``$+\ 1$"'s in the equations involving $\gamma$ come from counting $\gamma_j$, the second occurrence of $\gamma_1$.)

Combining all of the above observations, we get
\[ \gamma(s;+) + \gamma(s;-) + \gamma_F(s) - 1 = \tau(s;+) + \tau(s;-) + \tau_F(s). \]

If $\gamma(s;-) = \tau(s;-)$ and $\gamma(s;+) = \tau(s;+)$, this reduces to
\[ \gamma_F(s) = \tau_F(s) + 1, \]
which contradicts Lemma \ref{lem:3.2.A}.  Thus we cannot have both $\gamma(s;-) = \tau(s;-)$ and $\gamma(s;+) = \tau(s;+)$.
\end{proof}

\begin{lemma}\label{lem:3.2.C}
Either $\gamma(s;+) > \tau(s;+)$ for all $s \in [j,k-1]$, or $\gamma(s;-) > \tau(s;-)$ for all $s \in [j,k-1]$ (or both).
\end{lemma}
\begin{proof}
If $\gamma(s;+) > \tau(s;+)$ for all $s \in [j,k-1]$, we are done.  Otherwise, there is some first index $s \in [j,k-1]$ for which $\gamma(s;+) = \tau(s;+)$.  By Lemma \ref{lem:3.2.B}, we must have $\gamma(s;-) > \tau(s;-)$.  Since there are no $+$ jumps for $\gamma$ in the range $s+1,\hdots,k-1$ (by the fact that we are not in Case \ref{sc:1212}, and by our choice of $k$), the equality $\gamma(t;+) = \tau(t;+)$ must hold for all $t = s+1,\hdots,k-1$ as well, and thus, again by Lemma \ref{lem:3.2.B}, we have $\gamma(t;-) > \tau(t;-)$ for all $t$ in this range as well.  Since there are no $-$ jumps for $\gamma$ in the range $[j+1,s-1]$ (again, because we are not in Case \ref{sc:1212} and by our choice of $k$), the strict inequality $\gamma(t;-) > \tau(t;-)$ must hold also for all $t \in [j,s-1]$.  Indeed, assuming by induction that strict inequality holds at position $t$, we have
\[ \gamma(t-1;-) = \gamma(t;-) > \tau(t;-) \geq \tau(t-1;-). \]
\end{proof}

Now, using Lemma \ref{lem:3.2.C}, we are able to show that the $\gamma'$ obtained from $\gamma$ by replacing $(\gamma_1,\gamma_j,\gamma_k,\gamma_l) = 1122$ by either $1+-1$ or $1-+1$ satisfies $\gamma < \gamma' \leq \tau$.  If $\gamma(s;+) > \tau(s;+)$ for all $s \in [j,k-1]$, then we replace the pattern by $1-+1$, and if $\gamma(s;-) > \tau(s;-)$ for all $s \in [j,k-1]$, we replace the pattern by $1+-1$.  (If both of these are true, we can make either move.)

Assume that $\gamma(s;+) > \tau(s;+)$ for all $s \in [j,k-1]$.  Then $\gamma'$ is obtained via the move $1122 \rightarrow 1-+1$.  Note how the rank numbers of $\gamma'$ differ from those of $\gamma$:

\begin{itemize}
	\item $\gamma'(s;+) = \gamma(s;+) - 1$ for all $s \in [j,k-1]$;
	\item $\gamma'(s;t) = \gamma(s;t) + 1$ whenever $s<t$, $f \leq s < k$, and $j \leq t < l$.
\end{itemize}

This shows that $\gamma < \gamma'$.  To see that $\gamma' \leq \tau$, we must show that
\begin{itemize}
	\item $\gamma(s;-) > \tau(s;-)$ for all $s \in [j,k-1]$; and
	\item $\gamma(s;t) < \tau(s;t)$ whenever $s<t$, $f \leq s < k$, and $j \leq t < l$.
\end{itemize}

The first of these two items has already been assumed.  The second follows from Corollary \ref{cor:bruhat-involutions} and \cite[Corollary 4.6]{Incitti-04}.  Indeed, on the level of the underlying involutions, the move from $\gamma$ to $\gamma'$ is the minimal covering transformation of $\gamma$ relative to $\tau$, with $(1,k)$ being a ``crossing $ee$-rise".

Now, if it is not the case that $\gamma(s;+) > \tau(s;+)$ for all $s \in [j,k-1]$, then by Lemma \ref{lem:3.2.C}, it \textit{is} the case that $\gamma(s;-) > \tau(s;-)$ for all $s \in [j,k-1]$.  Thus we instead make the move $1122 \rightarrow 1+-1$, and repeat the above argument with signs reversed.
\end{subsubcase}

\begin{subsubcase}[$\gamma_k = +$.]
We have the following picture:
\[
\begin{array}{rcccccccc}
		& & & & & & & i & \\
		\tau: & 1 & \hdots & \hdots & \hdots & \hdots & \hdots & 1 & \hdots \\
		\gamma: & 1 & \hdots & 1 & \hdots & + & \hdots & \hdots & \hdots \\
		 & & & j & & k & & & 
	\end{array}
\]

Note that $(\gamma_1,\gamma_j,\gamma_k)$ form the pattern $11+$.  We would like to obtain $\gamma'$ from $\gamma$ by converting this pattern to $1+1$.  Alas, this does not always work.  The $\gamma'$ so obtained satisfies $\gamma < \gamma' \leq \tau$ in some cases, and does not in other cases.

To see this, note how the rank numbers for $\gamma$ differ from those of $\gamma'$:
\begin{itemize}
	\item $\gamma'(s;-) = \gamma(s;-) - 1$ for $s = j,\hdots,k-1$.
	\item $\gamma'(s;t) = \gamma(s;t) + 1$ for $s<t$ with $1 \leq s$, $j \leq t < k$.
\end{itemize}

Thus it is always the case that $\gamma < \gamma'$.  However, it is only true that $\gamma' \leq \tau$ if $\gamma(s;-) > \tau(s;-)$ for all $s \in [j,k-1]$.  If this holds, then we do indeed have that $\gamma' \leq \tau$.  (Again, the inequalities $\gamma'(s;t) \leq \tau(s;t)$ follow from Corollary \ref{cor:bruhat-involutions} and \cite[Corollary 4.6]{Incitti-04}, since, on the level of the underlying involutions, the move $11+ \rightarrow 1+1$ is the minimal covering transformation of $\gamma$ relative to $\tau$, with $(1,k)$ being an ``$ef$-rise".)

We need not have $\gamma(s;-) > \tau(s;-)$ for all $s \in [j,k-1]$, however.  In the event that we do not, then we claim that one of the following must hold:
\begin{enumerate}
	\item There exists $l \in [k+1,i]$ with $\gamma_l = -$, or
	\item There exists a pair of matching natural numbers $\gamma_l = \gamma_m \in \N$ with $k < l < m \leq i$.
\end{enumerate}

To see this, note that there is some $s \in [j,k-1]$ for which $\gamma(s;-) = \tau(s;-)$, and since $\gamma$ has no $-$ jumps in the range $[s,k]$ (by our choice of $k$), we also have that $\gamma(t;-) = \tau(t;-)$ for $t = s+1,\hdots,k$.  Now, $\tau$ has a $-$ jump at position $i$, since $\tau_i$ is the second occurrence of $\tau_1$.  Thus to ensure that $\gamma(i;-) \geq \tau(i;-)$, $\gamma$ must have a $-$ jump somewhere in the range $[k+1,i]$.  This can occur either with a $-$ sign, or with the second occurrence of a natural number.  If it occurs at the second occurrence of a natural number, the first occurrence of that number cannot be in the range $[2,j-1]$, since we are not in Case \ref{sc:1212}, and it cannot be in the range $[j+1,k]$, by our choice of $k$.  Thus the first occurrence of that number must be beyond position $k$.

Now, choose $l$ to be the smallest index in the range $[k+1,i]$ such that either $\gamma_l = -$ or $\gamma_l = \gamma_m \in \N$ with $k < l < m \leq i$.

\begin{subsubsubcase}[$\gamma_l = -$]
The picture is 
\[
\begin{array}{rccccccccc}
		& & & & & & & & i & \\
		\tau: & 1 & \hdots & \hdots & \hdots & \hdots & \hdots & \hdots & 1 & \hdots \\
		\gamma: & 1 & \hdots & 1 & \hdots & + & \hdots & - & \hdots & \hdots \\
		 & & & j & & k & & l & &
	\end{array}
\]

Then $(\gamma_1,\gamma_j,\gamma_l)$ form the pattern $11-$, and we claim that the $\gamma'$ obtained from $\gamma$ by converting this pattern to $1-1$ satisfies $\gamma < \gamma' \leq \tau$.  Note how the rank numbers for $\gamma$ and $\gamma'$ are related:

\begin{itemize}
	\item $\gamma'(s;+) = \gamma(s;+) - 1$ for $s = j,\hdots,l-1$.
	\item $\gamma'(s;t) = \gamma(s;t) + 1$ for $f \leq s$, $j \leq t < l$.
\end{itemize}

Thus $\gamma < \gamma'$.  To see that $\gamma' \leq \tau$, we need to see that 
\begin{enumerate}
	\item $\gamma(s;+) > \tau(s;+)$ for $s=j,\hdots,l-1$, and 
	\item $\gamma(s;t) < \tau(s;t)$ for $s < t$ with $1 \leq s$ and $j \leq t < l$.
\end{enumerate}

For the first, note that we have already assumed that $\gamma(s;-) = \tau(s;-)$ for some $s \in [j+1,k]$, so by Lemma \ref{lem:3.2.C}, $\gamma(s;+) > \tau(s;+)$ for all $s \in [j,k-1]$.  Since $\gamma_k = +$, $\gamma(k;+) > \tau(k;+)$ as well.  Now, since $\gamma(s;-) = \tau(s;-)$ and since there are no $-$ jumps in the range $[s,l]$ (by our choices of both $k$ and $l$), the equality $\gamma(t;-) = \tau(t;-)$ is maintained for all $t \in [s,l-1]$.  Then by Lemma \ref{lem:3.2.B}, we have $\gamma(t;+) > \tau(t;+)$ in all of these positions as well.  (Note that Lemma \ref{lem:3.2.B} refers to indices in the range $[j,k-1]$, but in fact this restriction on the indices is not necessary.  Indeed, the upper bound of the interval for which Lemma \ref{lem:3.2.B} holds can be taken to be any index less than $i$.)

Unlike several of our other cases, here the inequality $\gamma(s;t) < \tau(s;t)$ does not follow immediately from the results of \cite{Incitti-04} since the move being described here is not a ``minimal covering transformation".  Thus we must argue it directly here.  So let $s<t$ be given with $1 \leq s$ and $j \leq t < l$.  By our choice of $j$, $k$, $l$, etc., and by virtue of the case that we are currently in, it is clear that all first occurrences in the range $[1,s]$ either have their mate in a position at most $j$, or else strictly beyond $i$.  So since $t \geq j$, we have
\[ \gamma(s;t) = \#\{\gamma_a = \gamma_b \in \N \mid a \leq s, b > t\} = \#\{\gamma_a = \gamma_b \in \N \mid a \leq s, b > i\} = \gamma(s;i). \]

As for $\tau(s;t)$, since $t < l \leq i$, we have 
\[ \tau(s;t) = \#\{\tau_a = \tau_b \in \N \mid a \leq s, b > t\} > \#\{\tau_a = \tau_b \in \N \mid a \leq s, b > i\} = \tau(s;i), \]
since the pair $\tau_1 = \tau_i$ contributes to $\tau(s;t)$, but not to $\tau(s;i)$.  Since
\[ \gamma(s;t) = \gamma(s;i) \leq \tau(s;i) < \tau(s;t), \]
we are done.
\end{subsubsubcase}

\begin{subsubsubcase}[$\gamma_l = \gamma_m$ for $k < l < m \leq i$]
Here, the picture is
\[
\begin{array}{rccccccccccc}
		& & & & & & & & & & i & \\
		\tau: & 1 & \hdots & \hdots & \hdots & \hdots & \hdots & \hdots & \hdots & \hdots & 1 & \hdots \\
		\gamma: & 1 & \hdots & 1 & \hdots & + & \hdots & 2 & \hdots & 2 & \hdots & \hdots \\
		 & & & j & & k & & l & & m & 
	\end{array}
\]
Here, we claim that the $\gamma'$ obtained from $\gamma$ by changing the pattern $(\gamma_1,\gamma_j,\gamma_l,\gamma_m) = 1122$ to $1-+1$ satisfies $\gamma < \gamma' \leq \tau$.  Indeed, the arguments for this are identical to those given in the previous subcase.  Recall how the rank numbers for $\gamma$ and $\gamma'$ differ:

\begin{itemize}
	\item $\gamma'(s;+) = \gamma(s;+) - 1$ for all $s \in [j,l-1]$;
	\item $\gamma'(s;t) = \gamma(s;t) + 1$ whenever $s<t$, $1 \leq s < l$, and $j \leq t < m$.
\end{itemize}

Thus $\gamma < \gamma'$, and to see that $\gamma' \leq \tau$, we simply need to see that $\gamma(s;+) > \tau(s;+)$ for $s \in [j,l-1]$, and that $\gamma(s;t) < \tau(s;t)$ for $s<t$, $1 \leq s < l$, and $j \leq t < m$.  As mentioned, the arguments for these facts are identical to those given in the previous subcase, where $\gamma_l$ was equal to $-$.  Since they are identical, we do not repeat them here.
\end{subsubsubcase}

In conclusion, the result of all the subcases considered here can be stated succinctly as follows:  A suitable $\gamma'$ can be obtained from $\gamma$ through one of the following moves:  $11+ \rightarrow 1+1$, $11- \rightarrow 1-1$, or $1122 \rightarrow 1-+1$.
\end{subsubcase}

\begin{subsubcase}[$\gamma_k = -$.]
Here, we repeat the arguments of the previous subcase, reversing all signs.  The upshot is that $\gamma'$ can be obtained from $\gamma$ by one of the moves $11- \rightarrow 1-1$, $11+ \rightarrow 1+1$, or $1122 \rightarrow 1+-1$.
\end{subsubcase}
\end{subcase}

\begin{subcase}[None of the previous cases apply.]
Then by Lemma \ref{lem:3.B}, there exists a pair of matching natural numbers $\gamma_k = \gamma_l \in \N$ with $k \in [j+1,i]$ and $l > i$.  Choose the pair such that $l$ is minimal --- that is, choose $l$ to be the smallest index greater than $i$ which is the second occurrence of a natural number whose first occurrence is between indices $j+1$ and $i$.  We have the following schematic:
\[
\begin{array}{rcccccccccc}
		& & & & & & & i & \\
		\tau: & 1 & \hdots & \hdots & \hdots & \hdots & \hdots & 1 & \hdots & \hdots & \hdots\\
		\gamma: & 1 & \hdots & 1 & \hdots & 2 & \hdots & \hdots & \hdots & 2 & \hdots \\
		 & & & j & & k & & & & l & 
	\end{array}
\]

Note that the indices $(\gamma_1,\gamma_j,\gamma_k,\gamma_l)$ form the pattern $1122$.  We claim that the $\gamma'$ obtained from $\gamma$ by converting this pattern to $1212$ (i.e. by interchanging $\gamma_j$ and $\gamma_k$) satisfies $\gamma < \gamma' \leq \tau$.  To see that $\gamma < \gamma'$, note how the rank numbers for $\gamma$ and $\gamma'$ are related.  The differences are:

\begin{itemize}
	\item $\gamma'(s;+) = \gamma(s;+) - 1$ for $s = j,\hdots,k-1$;
	\item $\gamma'(s;-) = \gamma(s;-) - 1$ for $s=j,\hdots,k-1$;
	\item $\gamma'(s;t) = \gamma(s;t) + 1$ for $f \leq s < j \leq t < k$;
	\item $\gamma'(s;t) = \gamma(s;t) + 1$ for $j \leq s < k \leq t < l$;
	\item $\gamma'(s;t) = \gamma(s;t) + 2$ for $j \leq s < t < k$.
\end{itemize}

This shows that $\gamma < \gamma'$.  To see that $\gamma' \leq \tau$, we must show that 
\begin{itemize}
	\item $\gamma(s;\pm) > \tau(s;\pm)$ for $s = j,\hdots,k-1$;
	\item $\gamma(s;t) > \tau(s;t)$ for $f \leq s < j \leq t < k$;
	\item $\gamma(s;t) < \tau(s;t)$ for $j \leq s < k \leq t < l$;
	\item $\gamma(s;t) +1 < \tau(s;t)$ for $j \leq s < t < k$.
\end{itemize}

The last three items above follow from Corollary \ref{cor:bruhat-involutions} and \cite[Corollary 4.6]{Incitti-04}.  Indeed, on the level of underlying involutions, the move from $\gamma$ to $\gamma'$ is the minimal covering transformation of $\gamma$ relative to $\tau$, with $(1,l)$ being an ``$ed$-rise".

So we concern ourselves only with the first item.  Note that, by virtue of the case we currently find ourselves in, $\gamma$ has no $+$ or $-$ jumps in the range $j+1,\hdots,i$.  Since $\tau$ \textit{does} have both a $+$ and $-$ jump at position $i$ ($\tau_i$ being the second occurrence of $\tau_1$), and since we must have $\tau(i;\pm) \leq \gamma(i;\pm)$, we see that we must have $\gamma(s;\pm) > \tau(s;\pm)$ for all $s$ in the range $j+1,\hdots,i-1$.  Since $k \leq i$, in particular we have that $\gamma(s;\pm) > \tau(s;\pm)$ for $s = j+1,\hdots,k-1$.  Since $\gamma$ has both a $+$ and a $-$ jump at position $j$ ($\gamma_j$ being the second occurrence of $\gamma_f$), we must have that $\gamma(j;\pm) > \tau(j;\pm)$ as well.
\end{subcase}
\end{case}

This completes the proof of Theorem \ref{thm:covers-combinatorial-order}.
\end{proof}

With Theorem \ref{thm:covers-combinatorial-order} in hand, we can now prove Theorem \ref{thm:bruhat}.

\begin{proof}[Proof of Theorem \ref{thm:bruhat}]
Let $\gamma$,
$\tau$ be $(p,q)$-clans, with $Q_{\gamma}$, $Q_{\tau}$ the corresponding $K$-orbits, and $Y_{\gamma}$, $Y_{\tau}$ the
corresponding $K$-orbit closures.  We prove the first statement of Theorem \ref{thm:bruhat}, namely that $Y_{\gamma}
\subseteq Y_{\tau}$ if and only if $\gamma \leq \tau$.

We show first that $\gamma \nleq \tau \Rightarrow Y_{\gamma} \not\subset Y_{\tau}$.  Let $\mathcal{Y}_{\tau}$ be the
purported orbit closure described in Theorem \ref{thm:bruhat}, defined by inequalities determined by the rank
numbers of $\tau$.  $\mathcal{Y}_{\tau}$ is a closed subvariety of $G/B$, being defined locally by the vanishing of
certain minors.  Indeed, conditions (1) and (2) of the description of $\mathcal{Y}_{\tau}$ amount to the vanishing of
lower-left $i \times q$ minors and upper-left $i \times p$ minors, respectively, of a generic matrix whose
non-specialized entries give affine coordinates on a translated big cell.  Condition (3) amounts to the vanishing of
certain minors of a matrix whose first $i$ columns are the first $i$ columns of this generic matrix with the
lower-left $i \times p$ submatrix zeroed out, and whose last $j$ columns are the first $j$ columns of the generic
matrix.  (These conditions are explained more precisely in \cite{Wyser-Yong-13}.)  Since $Q_{\tau} \subset \mathcal{Y}_{\tau}$ by Theorem \ref{thm:orbit_description}, we clearly have that
$Y_{\tau} \subseteq \mathcal{Y}_{\tau}$.  Thus to show that $Y_{\gamma} \not\subset Y_{\tau}$, it suffices to show
that $Q_{\gamma} \cap \mathcal{Y}_{\tau} = \emptyset$.  This is clear from the definitions, since if
$\gamma \nleq \tau$, there exists some $i$ such that $\gamma(i;+) < \tau(i;+)$, or $\gamma(i;-) < \tau(i;-)$, or some
$i<j$ such that $\gamma(i;j) > \tau(i;j)$.  Since any point of $Q_{\gamma}$ must meet the description of Theorem
\ref{thm:orbit_description}, it cannot possibly lie in $\mathcal{Y}_{\tau}$.

The preceding argument establishes that $Y_{\gamma} \subseteq Y_{\tau} \Rightarrow \gamma \leq \tau$.  We now consider
the converse.  Suppose that $\gamma < \tau$.  By induction, it suffices to consider the cases where $\tau$ covers $\gamma$, so we
may assume that $\gamma$ and $\tau$ are related by a ``move" of the type described in Theorem
\ref{thm:covers-combinatorial-order}.  For each of those possible moves, we give the following sort of argument:  We
take representatives $F_{\bullet} \in Q_{\tau}$ and $E_{\bullet} \in Q_{\gamma}$, provided by the algorithm of
\cite{Yamamoto-97} described in Subsection \ref{sec:yamamoto-stuff}.  We also take, for each $t \in \C^*$, a matrix
$k(t) \in K$, so that the flag $F_{\bullet}(t) := k(t) \cdot F_{\bullet}$ is in $Q_{\tau}$ for all $t \in \C^*$.  We
then show that $\lim_{t \rightarrow 0} F_{\bullet}(t) = E_{\bullet}$.  This shows that $E_{\bullet}$ is a limit point
for $Q_{\tau}$, so that it is an element of $Y_{\tau}$.  Now, any other point $P_{\bullet} \in Q_{\gamma}$ is of the
form $P_{\bullet} = k \cdot E_{\bullet}$ for some $k \in K$.  Moreover, the curve $k \cdot F_{\bullet}(t)$ is
contained in $Q_{\tau}$, and tends to $P_{\bullet}$ as $t \rightarrow 0$.  So this argument establishes that in fact
$Q_{\gamma} \subseteq Y_{\tau}$, which implies that $Y_{\gamma} \subseteq Y_{\tau}$.

For ease of notation, when writing the flags $E_{\bullet} = \left\langle v_1,\hdots,v_n \right\rangle$ and $F_{\bullet} = \left\langle w_1,\hdots,w_n \right\rangle$, we may indicate only those $v_i$ and $w_i$ which differ.  For instance, if $E_{\bullet} = \left\langle e_1, e_2, e_3, e_4, e_5 \right\rangle$ and $F_{\bullet} = \left\langle e_1, e_2+e_3, e_3, e_4, e_5 \right\rangle$, then for short we will write $E_{\bullet} = \left\langle e_2 \right\rangle$ and $F_{\bullet} = \left\langle e_2 + e_3 \right\rangle$.
\setcounter{case}{0}
\begin{case}[$+- \rightarrow 11$]
Suppose that the $+$ and $-$ are in positions $i < j$.  We may choose the representatives $E_{\bullet} = \left\langle v_1,\hdots,v_n \right\rangle$ and $F_{\bullet} = \left\langle w_1,\hdots,w_n \right\rangle$ so that $v_i = e_1$, $v_j = e_n$, $w_i = e_1 + e_n$, $w_j = e_n$, and $v_l = w_l$ for all remaining $l$.  Now, for $t \in \C^*$, let $k(t) \in K$ be the diagonal matrix
\[ k(t) = \text{diag}(1/t,1,1,\hdots,1). \]
Then 
\[ F_{\bullet}(t) := k(t) \cdot F_{\bullet} = \left\langle (1/t) e_1 + e_n \right\rangle = \left\langle e_1 + te_n \right\rangle. \] 
(To obtain the last equality, we have simply scaled the $i$th basis vector for $F_{\bullet}(t)$ by a factor of $t$.)  From this, it is clear that $\lim_{t \rightarrow 0} F_{\bullet}(t) = E_{\bullet}$.
\end{case}
\begin{case}[$-+ \rightarrow 11$]
This case is extremely similar to the previous case.  We omit the details.
\end{case}
\begin{case}[$11+ \rightarrow 1+1$]
Suppose the $11+$ and $1+1$ occur in positions $i<j<k$.  We may choose $E_{\bullet} = \left\langle v_1,\hdots,v_n \right\rangle$ so that $v_i = e_1 + e_n$, $v_j = e_1$, and $v_k = e_2$.  We write
\[ E_{\bullet} = \left\langle e_1 + e_n, e_1, e_2 \right\rangle = \left\langle e_1+e_n, e_1, 2e_1 + e_2 + e_n \right\rangle, \]
where we have simply replaced $v_k$ by $v_i + v_j + v_k$, which does not change the point $E_{\bullet}$.

Likewise, we may choose $F_{\bullet} = \left\langle w_1,\hdots,w_n \right\rangle$ so that $w_i = e_1 + e_n$, $w_j = e_2$, and $w_k = e_1$.  Then
\[ F_{\bullet} = \left\langle e_1+e_n,e_2,e_1 \right\rangle = \left\langle e_1 + e_n, e_2, 2e_1 + e_2 + e_n \right\rangle. \]

Finally, choose $k(t) \in K$ to be the matrix with $1$'s on the diagonal, $1/t$ in entry $(1,2)$, and $0$'s elsewhere.  Then 
\[ F_{\bullet}(t) = k(t) \cdot F_{\bullet} = \left\langle e_1 + e_n, (1/t) e_1 + e_2, (2+(1/t)) e_1 + e_2 + e_n \right\rangle =^* \]
\[ \left\langle e_1 + e_n, (1/t) e_1 + e_2, 2e_1 + e_2 + e_n \right\rangle = \]
\[ \left\langle e_1 + e_n, e_1 + te_2, 2e_1 + e_2 + e_n \right\rangle. \]

(Note that to obtain the equality (*), we simply replace $(2+(1/t))e_1 + e_2 + e_n$ by 
\[ (1-(1/t)) ((2+(1/t)) e_1 + e_2 + e_n) + (1/t) ((1/t) e_1 + e_2) + (1/t) (e_1 + e_n), \]
which does not change the flag $F_{\bullet}(t)$.)

By the last description of $F_{\bullet}(t)$ given above, we see that $E_{\bullet} = \lim_{t \rightarrow 0} F_{\bullet}(t)$.
\end{case}
\begin{case}[$+11 \rightarrow 1+1$]
Suppose that $+11$ and $1+1$ occur in positions $i<j<k$.  We may choose $E_{\bullet} = \left\langle v_1,\hdots,v_n \right\rangle$ so that $v_i = e_1$, $v_j = e_2+e_n$, and $v_k = e_n$.  We write
\[ E_{\bullet} = \left\langle e_1, e_2+e_n, e_n \right\rangle = \left\langle e_1, e_1 + e_2 + e_n, e_n \right\rangle. \]

We may choose $F_{\bullet} = \left\langle w_1,\hdots,w_n \right\rangle$ so that $w_i = e_1 + e_n$, $w_j = e_2$, and $w_k = e_n$.  Then
\[ F_{\bullet} = \left\langle e_1+e_n,e_2,e_n \right\rangle = \left\langle e_1 + e_n, e_1 + e_2 + e_n, e_n \right\rangle. \]

Finally, choose $k(t) \in K$ to be the matrix with $1/t$ in entry $(1,1)$, $1$'s on all of the other diagonal entries, $1-(1/t)$ in entry $(1,2)$, and $0$'s elsewhere.  Then
\[ F_{\bullet}(t) = \left\langle (1/t) e_1 + e_n, e_1 + e_2 + e_n, e_n \right\rangle = \left\langle e_1 + te_n, e_1 + e_2 + e_n, e_n \right\rangle, \]
from which it is clear that $E_{\bullet} = \lim_{t \rightarrow 0} F_{\bullet}$.

\end{case}
\begin{case}[$11- \rightarrow 1-1$, $-11 \rightarrow 1-1$]
These cases are extremely similar to the previous two cases, so we omit the details.
\end{case}
\begin{case}[$1122 \rightarrow 1212$]
Suppose that the $1122$ and $1212$ patterns occur in positions $i < j < k < l$.  We choose the flag $E_{\bullet} = \left\langle v_1,\hdots,v_n \right\rangle$ so that $v_i = e_1 + e_{n-1}$, $v_j = e_{n-1}$, $v_k = e_2 + e_n$, and $v_l = e_n$.  We choose the flag $F_{\bullet} = \left\langle w_1,\hdots,w_n \right\rangle$ so that $w_i = e_1 + e_{n-1}$, $w_j = e_2 + e_n$, $w_k = e_{n-1}$, and $w_l = e_n$.  Let $k(t) \in K$ be the matrix with $1$'s on the diagonal, $1/t$ in position $(n-1,n)$, and $0$'s elsewhere.  Then
\[ F_{\bullet}(t) = \left\langle e_1 + e_{n-1}, e_2 + (1/t) e_{n-1} + e_n, e_{n-1}, (1/t) e_{n-1} + e_n \right\rangle =^* \]
\[ \left\langle e_1 + e_{n-1}, e_2 + (1/t) e_{n-1} + e_n, e_{n-1}, e_n \right\rangle =^{**} \]
\[ \left\langle e_1 + e_{n-1}, e_2 + (1/t) e_{n-1} + e_n, e_2 + e_n, e_n \right\rangle = \]
\[ \left\langle e_1 + e_{n-1}, e_{n-1} + t(e_2+e_n), e_2 + e_n, e_n \right\rangle. \]

Note that the equality (*) is obtained by replacing $(1/t) e_{n-1} + e_n$ by $((1/t) e_{n-1} + e_n) - (1/t) (e_{n-1})$, while the equality (**) is obtained by replacing $e_{n-1}$ by $(-1/t) (e_{n-1}) + (e_2 + (1/t) e_{n-1} + e_n)$, neither of which changes the flag $F_{\bullet}(t)$.  By the last description of $F_{\bullet}(t)$ given above, we see that $E_{\bullet} = \lim_{t \rightarrow 0} F_{\bullet}(t)$.
\end{case}

\begin{case}[$1122 \rightarrow 1+-1$]
Suppose that the $1122$ and $1+-1$ occur in positions $i < j < k < l$.  Choose the representative $E_{\bullet} = \left\langle v_1,\hdots,v_n \right\rangle$ so that $v_i = e_1 + e_{n-1}$, $v_j = e_{n-1}$, $v_k = e_2 + e_n$, and $v_l = e_n$.  Choose the flag $F_{\bullet} = \left\langle w_1,\hdots, w_n \right\rangle$ so that $w_i = e_1+e_n$, $w_j = e_{n-1}$, $w_k = e_2$, and $w_l = e_n$.  We may rewrite $F_{\bullet} = \left\langle w_i, w_j, w_k, w_l \right\rangle$ as $\left\langle e_1+e_n, e_{n-1}, e_1+e_2+e_n, e_n \right\rangle$.

Let $k(t) \in K$ be the matrix whose upper-left $2 \times 2$ corner is 
\[ \begin{pmatrix}
		1/t & -1/t \\
		1 & 0
	\end{pmatrix},
\]
whose lower-right $2 \times 2$ corner is
\[ \begin{pmatrix}
		-1/t & 1/t \\
		0 & 1
	\end{pmatrix},
\]
and which has $1$'s in all other diagonal entries, and $0$'s elsewhere.  Then 
\[ F_{\bullet}(t) = k(t) \cdot F_{\bullet} = \left\langle (1/t) e_1 + e_2 + (1/t) e_{n-1} + e_n, e_2 + (1/t)e_{n-1} + e_n, e_2 + e_n, (1/t) e_{n-1} + e_n \right\rangle =^* \]
\[ \left\langle (1/t)e_1 + e_2 + (1/t) e_{n-1} + e_n, e_2 + (1/t) e_{n-1} + e_n, e_2 + e_n, e_n \right\rangle = \]
\[ \left\langle e_1 + e_{n-1} + t(e_2 + e_n), e_{n-1} + t(e_2+e_n), e_2 + e_n, e_n \right\rangle, \]
where the equality (*) is obtained by replacing $(1/t) e_{n-1} + e_n$ by
\[ (-1)(e_2+(1/t) e_{n-1} + e_n) + (1)(e_2 + e_n) + (1)((1/t) e_{n-1} + e_n). \]
From the last description of $F_{\bullet}(t)$ above, we see that $\lim_{t \rightarrow 0} F_{\bullet}(t) = E_{\bullet}$.
\end{case}
\begin{case}[$1122 \rightarrow 1-+1$]
This case is extremely similar to the last one, except a bit simpler, so we omit the details.
\end{case}
\begin{case}[$1212 \rightarrow 1221$]
Suppose that the $1212$ and $1221$ occur in positions $i < j < k < l$.  Choose the representative $E_{\bullet} = \left\langle v_1,\hdots,v_n \right\rangle$ so that $v_i = e_1 + e_{n-1}$, $v_j = e_2 + e_n$, $v_k = e_{n-1}$, and $v_l = e_n$.  Choose the flag $F_{\bullet} = \left\langle w_1,\hdots, w_n \right\rangle$ so that $w_i = e_1+e_n$, $w_j = e_2 + e_{n-1}$, $w_k = e_{n-1}$, and $w_l = e_n$.  We may rewrite $F_{\bullet} = \left\langle w_i, w_j, w_k, w_l \right\rangle$ as $\left\langle e_1+e_n, e_1+e_2+e_{n-1}+e_n, e_{n-1}, e_n \right\rangle$.

Let $k(t) \in K$ be the same matrix described in the case $1122 \rightarrow 1+-1$.  Then 
\[ F_{\bullet}(t) = k(t) \cdot F_{\bullet} = \] 
\[ \left\langle (1/t) e_1 + e_2 + (1/t) e_{n-1} + e_n, e_2 + e_n, (-1/t) e_{n-1}, (1/t) e_{n-1} + e_n \right\rangle = \]
\[ \left\langle (1/t) e_1 + e_2 + (1/t) e_{n-1} + e_n, e_2 + e_n, e_{n-1}, e_n \right\rangle = \]
\[ \left\langle e_1 + e_{n-1} + t(e_2 + e_n), e_2 + e_n, e_{n-1}, e_n \right\rangle. \]
From the last description of $F_{\bullet}(t)$, we see that $\lim_{t \rightarrow 0} F_{\bullet}(t) = E_{\bullet}$.
\end{case}

This concludes the proof of Theorem \ref{thm:bruhat}.
\end{proof}

\begin{proof}[Proof of Corollary \ref{cor:orbit-closure}]
   We now deduce from Theorems \ref{thm:bruhat} and \ref{thm:orbit_description} the set-theoretic description of $K$-orbit
   closures given by Corollary \ref{cor:orbit-closure}.  
	Given a clan $\tau$, let $\mathcal{Y}_{\tau}$ denote the purported closure of $Q_{\tau}$, described in the
	statement of Corollary \ref{cor:bruhat}, and let $Y_{\tau} = \overline{Q_{\tau}}$ be the true closure.  Clearly, 
	$Y_{\tau}$ is $K$-stable, and hence is a union of $K$-orbits.  Namely, it is the union of all $K$-orbits
	$Q_{\gamma}$ which are contained in $Y_{\tau}$.  Thus 
	\[ Y_{\tau} = \bigcup_{Q_{\gamma} \subset Y_{\tau}} Q_{\gamma} = \bigcup_{\gamma \leq \tau} Q_{\gamma}. \]
	
	Suppose that $F_{\bullet} \in Q_{\gamma}$ for some $\gamma \leq \tau$.  Then, by Theorem \ref{thm:orbit_description}, 
	\begin{itemize}
		\item $\dim(F_i \cap E_p) = \gamma(i;+) \geq \tau(i;+)$ for $i=1,\hdots,n$.
		\item $\dim(F_i \cap \widetilde{E_q}) = \gamma(i;-) \geq \tau(i;-)$ for $i=1,\hdots,n$.
		\item $\dim(\pi(F_i) + F_j) = j + \gamma(i;j) \leq j + \tau(i;j)$ for $1 \leq i < j \leq n$.
	\end{itemize}
	
	Thus $Y_{\tau} \subseteq \mathcal{Y}_{\tau}$.
	
	For the other inclusion, a flag $F_{\bullet} \in \mathcal{Y}_{\tau}$ is clearly in \textit{some} $K$-orbit $Q_{\gamma}$, and since 
	\begin{itemize}
		\item $\gamma(i;+) = \dim(F_i \cap E_p) \geq \tau(i;+)$ for $i=1,\hdots,n$;
		\item $\gamma(i;-) = \dim(F_i \cap \widetilde{E_q}) \geq \tau(i;-)$ for $i=1,\hdots,n$;
		\item $j + \gamma(i;j) = \dim(\pi(F_i) + F_j) \leq j + \tau(i;j)$ for $1 \leq i < j \leq n$,
	\end{itemize}
	we have $\gamma \leq \tau$.  Thus $F_{\bullet} \in Y_{\tau}$.  We conclude that $\mathcal{Y}_{\tau} = Y_{\tau}$.
\end{proof}

\section*{Acknowledgments}
The author was supported by NSF International Research Fellowship 1159045 and hosted by Institut Fourier in Grenoble,
France.

\bibliographystyle{plain}
\bibliography{../sourceDatabase}

\end{document}